\documentclass[a4paper,11pt,reqno]{amsart}
\usepackage{graphicx}
\graphicspath{ {images/} }
\usepackage{amssymb}
\usepackage{tikz-cd}
\usepackage{amsmath}
\usepackage{dsfont}
\usepackage{shuffle}
\usepackage{afterpage}
\usepackage{setspace}
\usepackage[top=1.2in, bottom=1.2in]{geometry}
\graphicspath{ {/Users/emanuele/Documents} }
\usepackage[utf8]{inputenc}

\usepackage[all]{xy}

\usepackage[lite]{amsrefs}

\newtheorem{thm}{Theorem}[section]
\newtheorem{pro}[thm]{Proposition}
\newtheorem{lem}[thm]{Lemma}
\newtheorem{cor}[thm]{Corollary}

\theoremstyle{definition}
\newtheorem{df}[thm]{Definition}

\theoremstyle{remark}
\newtheorem{rmk}[thm]{Remark}

\newtheorem{ex}[thm]{Example}

\allowdisplaybreaks

\usepackage{physics}
\newcommand\R{\mathbb{R}}
\newcommand\Z{\mathbb{Z}}

\title{Continuous cohomology of topological quandles}

\author{Mohamed Elhamdadi}
\address{Department of Mathematics and Statistics,
University of South Florida, Tampa, FL 33620, U.S.A.}
\email{emohamed@math.usf.edu}

\author{Masahico Saito}
\address{Department of Mathematics and Statistics,
	University of South Florida, Tampa, FL 33620, U.S.A.}
\email{saito@usf.edu}

\author{Emanuele Zappala}
\address{Department of Mathematics and Statistics,
	University of South Florida, Tampa, FL 33620, U.S.A.}
\email{zae@mail.usf.edu}

\begin{document}

\maketitle

\begin{abstract}
A continuous cohomology theory for topological quandles is introduced, and compared to the algebraic theories.
Extensions of topological quandles  are studied with respect to continuous 2-cocycles, and used to show the differences
in second cohomology
groups for specific topological quandles.
A method of computing the cohomology groups of the inverse limit is applied to quandles.

\end{abstract}

\tableofcontents

\section{Introduction}
Quandles are
non-associative algebraic structures whose study was motivated, in part, by knot theory.  Quandles and racks arose in many different areas of mathematics and appeared in the literature with different names since 1940's (e.g.~\cite{Takasaki}).
In 1982,  Joyce \cite{Joyce} (used the term quandle) and Matveev \cite{Matveev} (who call them distributive groupoids)  introduced independently the notion of a knot quandle $Q(K)$
associated to each oriented knot $K$.
The knot quandle is a complete invariant up to reversed mirror.
Since then quandles have been investigated
for constructing knot and link invariants
(e.g.~\cite{CES, CJKLS, EM}).
Quandles have been also studied
from different point of views such as the study of the set-theoretic Yang-Baxter equation \cite{CES2} and pointed Hopf algebras \cite{AG}.   Topological quandles were introduced in \cite{Rubin} and investigated further in \cite{ChengElhamdadiShekhtman, ClarkSaito, EM}.  Smooth quandles and their cohomology have been introduced by Nosaka~\cite{Nosaka}.

Homology theories of racks \cite{FRS} and quandles \cite{CJKLS} have been defined and investigated algebraically. They have also been used for constructing knot invariants called quandle cocycle invariants.
Continuous cohomology of topological groups  have been studied extensively (see \cite{Stasheff} for an overview), and the problem has been posed how quandle cohomology theories can be extended to topological quandles.

In the present article, a continuous cohomology theory for topological quandles is introduced, and compared to the algebraic theories.
 Extensions of topological quandles are studied with respect to continuous 2-cocycles, and used to show  differences in second cohomology
groups for specific topological quandles,
and to show non-triviality of continuous cohomology groups.
The formula of computing cohomology groups for inverse limits
have been used in various contexts (see  \cite{Wilson} for the case of groups).
We investigate this method in our context of quandles, and
apply it to concrete families of quandles to determine
cohomology groups of their inverse limits.

The following is the organization of the paper.  Section~\ref{Prel} gives the preliminary background on quandles.  In Section~\ref{secCohomology}, a continuous cohomology is introduced for topological quandles, and    the  first cohomology groups of topological Alexander quandles is investigated.
Furthermore
 extensions of  quandles by continuous 2-cocycles is studied  in the topological context.
 Section~\ref{InvLim} discusses the notion of inverse and direct limits of quandles and their cohomology groups, with examples presented.
Further examples of non-trivial cohomology groups are presented in Appendix~\ref{sec:nontriv},  and
the results that continuous isomorphism classes differ from algebraic isomorphism classes for
topological quandles are presented in
Appendix~\ref{ContinIsom},
that were obtained by W.~Edwin Clark.

\section{Preliminary} \label{Prel}

In this section we review basic definitions and properties of quandles.
A rack is a set $X$, together with a binary operation $ *: X \times X \longrightarrow X$ satisfying the following axioms:
\begin{itemize}
\item[(i)] for all $x,y\in X$, there exists a unique $z\in X$ such that  $z * x = y$;
\item[(ii)] for all $x,y$ and $z$ in $X$: $(x*y)*z = (x*z)*(y*z)$.
\end{itemize}
Property (ii) is commonly referred to as right self-distributivity.
Property (i)
states that the map $R_x$, defined by right $*$-multiplication by the element $x\in X$, is a bijection of $X$ onto itself.
A {\it quandle}  is an idempotent rack:
\begin{itemize}
\item[(iii)]
for all $x\in X$, $x*x = x$.
\end{itemize}

 A {\it quandle homomorphism} between two quandles $X, Y$ is
 a map $f: X \rightarrow Y$ such that $f(x*_X y)=f(x) *_Y f(y) $, where
 $*_X$ and $*_Y$
 denote
 the quandle operations of $X$ and $Y$, respectively.
   The subgroup of ${\rm Sym}(X)$ generated by the permutations ${ R}_x$, $x \in X$, is
called the {\it {\rm inner} automorphism group} of $X$,  and is
denoted by ${\rm Inn}(X)$.
A quandle is {\it indecomposable} (we use this term instead of {\it connected} to avoid confusion with topological meaning) if ${\rm Inn}(X)$ acts transitively on $X$.

A topological space $X$ with a continuous binary operation $(x,y) \mapsto x*y$ satisfying right self-distributivity and such that $R_x$ is a homeomorphism for all $x \in X$ is called a {\it topological rack}.   If moreover, the operation $*$ is idempotent then $X$ is called a {\it topological quandle}.
 Any rack $(X,*)$ can be regarded as a topological rack if $X$ is endowed with the discrete topology.

\begin{ex}
Typical examples of quandles arise as topological quandles when considered with topological structures as follows.

 Any topological group $G$ becomes a topological quandle with operation $*$ given by conjugation: $x*y = y^{-1} x y$. This quandle is %usually
 denoted ${\rm Conj}(G)$, the conjugation quandle associated to $G$.

For a topological group $G$ and a continuous automorphism $f: G \rightarrow G$,
$x*y=f(xy^{-1})y$ for $x, y \in G$ defines a topological quandle structure on $G$. This is called a {\it generalized Alexander quandle}
and is denoted by $(G, f)$. If $G$ is abelian it is called an Alexander quandle.

In particular,
for any $T \in {\rm GL}(n, \mathbb{R})$,
$\mathbb{R}^{n}$ can be given a topological quandle structure by defining $x*y = T x + (I - T)y$, for all $x,y \in \mathbb{R}^{n}$,
where $I$ denotes the identity matrix.
\end{ex}

\begin{ex}\label{ex:Rubin}
 Consider the $n$-dimensional sphere $\mathbb{S}^{n} \subset \mathbb{R}^{n+1}$. The operation $x*y = 2(x\cdot y)y - x$, for all $x,y \in \mathbb{S}^{n}$ endows the sphere with a topological quandle structure, where
 $x\cdot y$ denotes the standard inner product in $\mathbb{R}^{n+1}$.
Also, this operation induces a topological quandle structure on the real projective space $\mathbb{RP}^{n}$ giving a topological quandle homomorphism $S^n \rightarrow  \mathbb{RP}^{n}$.  Knot colorings by these quandles have been studied in \cite{Rubin}.

Similarly
for a topological abelian group $G$, %we can define
a topological quandle structure on $G$ is defined by the rule $x*y = 2y - x$. These
are a generalization of what is commonly referred to as Takasaki quandles \cite{Takasaki}.
\end{ex}

For more informations about quandles, discrete and topological, the interested reader can consult \cite{Joyce, Matveev, CJKLS, EN, EM, Rubin}.

\section{Continuous Cohomology of Topological Quandles} \label{secCohomology}

In this section we
define a {\it continuous cohomology theory} for topological quandles.
A similar theory for  smooth quandles was also defined and studied by Nosaka \cite{Nosaka} independently.
Rack and quandle homology theories have been defined and studied in several different contexts and have been generalized
(\cite{AG,CES,CJKLS}, for example).
In this section we follow \cite{CES} and in Section~\ref{sec:AG} we discuss %introduce
a generalization of the quandle cohomology defined in \cite{AG} for topological quandles.

Let $X$ be a topological quandle.
Let  $A$ be a topological abelian group, $T: A \rightarrow A$ be a continuous automorphism, and $A$ is also considered with the generalized Alexander quandle structure  $(A, T)$.
Consider the following abelian groups:
\[ \Gamma^n(X,A) = \{ f: X^n \rightarrow A\ |\ f\  \text{is continuous}, \ f(x_1, \cdots, x_n)=0 \; \text{if} \;  x_i=x_{i+1} \; \text{for some i} \ \}, \]
where $X^n$ is given the product topology induced by $X$  and the sum in $\Gamma^n$ is induced by pointwise addition in $A$.
Define  the following maps $\Gamma^n(X,A) \longrightarrow \Gamma^{n+1}(X,A)$, $n \in \mathbb{N}$:
\begin{eqnarray*}
& &  \delta^i_0 f(x_1, \ldots , x_{n+1}) = f(x_1, \ldots , \hat{x_i} , \ldots , x_{n+1}); \\
& & \delta^i_1 f(x_1, \ldots , x_{n+1}) = f(x_1*x_i, \ldots ,x_{i-1}*x_{i},x_{i+1},\ldots , x_{n+1}).
\end{eqnarray*}
We  define the chain complex:
\[\cdots \rightarrow \Gamma^n(X,A) \xrightarrow{\delta} \Gamma^{n+1}(X,A)\rightarrow\cdots\]
by setting:
\[ \delta^n =
\sum_{i=1}^{n+1}(-1)^i [ T \delta^i_0 - \delta^i_1 ] . \]
We  define the $n^{th}$ continuous cohomology group of $X$ with coefficients in $A$ by
\[ H^n_{\rm TC}(X,A) = \frac{ {\rm ker} (\delta^n)}{ {\rm im} (\delta^{n-1})},\]
assuming that the map $\delta^0$ is defined to be the canonical inclusion of the trivial group into $\Gamma^1(X,A)$, i.e. $H^1_{TC}(X,A) = \Gamma^1(X,A)$.

When $T=1$, the groups $H^n_{TC}(X,A)$ are called (untwisted)
 continuous quandle cohomology groups and will be denoted $H^n_{C}(X,A)$.
Furthermore, in Section~\ref{sec:AG}, a generalization of the twisted cohomology theory using quandle modules defined in \cite{AG} is defined and denoted by $H^n_{\rm GC}$.

Any topological quandle can be viewed as a discrete quandle, by forgetting the topological structure; it is clear that the chain complex above is a subcomplex of the usual (discrete) chain complex associated to $X$.
The original  and twisted quandle cohomology groups were denoted by $H_{\rm Q}$ and $H_{\rm T}$.
Our notation is summarized as follows.

\begin{center}
\begin{tabular}{llllll}
$H_{\rm Q}$ & : & Original (untwisted) & \quad
$H_{\rm T}$ & : & Original twisted \\
$H_{\rm C}$ & : & Continuous (untwisted) & \quad
$H_{\rm TC}$ & : & Continuous twisted \\
$H_{\rm GC}$ & : & Continuous generalized (quandle module)
\end{tabular}
\end{center}

\begin{ex} \label{ex:12}
Let $X$ be a topological quandle and $(A,T)$ be a topological Alexander quandle. Then
a continuous map $\eta: X \rightarrow A$ is a continuous 1-cocycle
if it satisfies
$T[\eta(y) - \eta (x) ] - [\eta (y) - \eta(x*y) ] =0$, that is, $\eta$ is a continuous quandle homomorphism,
$\eta(x*y)=T \eta (x) + (1-T) \eta (y)$.
If, in particular,  $T=1$ and $A$ is considered as a topological abelian group with trivial quandle structure, then the 1-cocycle condition is
$\eta (x*y) = \eta(x)$.

	A continuous map $\phi: X^2 \rightarrow A$ is a  2-cocycle if and only if it satisfies the condition:
	\[
	T\phi(x_1,x_2) + \phi(x_1*x_2,x_3) = \\
	T\phi(x_1,x_3) + (1-T)\phi(x_2,x_3) +  \phi(x_1*x_3,x_2*x_3)
	\]
    and $\phi(x,x)=0$.
	These considerations appear in \cite{CES} except the requirement of continuity.
\end{ex}

\subsection{Continuous First Cohomology}

Let $X$ be a topological quandle and $(A, T)$ be a topological Alexander quandle.
As mentioned in Example~\ref{ex:12}, it follows from the definition that $H^1_{TC}(X,A)$ is the group of
continuous quandle homomorphisms $\eta: X \rightarrow A$.

As a special case of $T=1$, where $A$ is regarded as a topological abelian group with the trivial quandle structure,
we have the following.
\begin{pro}
Let $X$ and $A$ be as above. If $X$ is indecomposable, then  the first cohomology group $H^1_{C}(X,A)$ is isomorphic  to $A$.
\end{pro}

\begin{proof}
The proof is similar to the original case, and our purpose here to include the proof is to observe that  the continuity does not play a role in the argument below as the 1-cocycle condition implies functions being constant.

Let $x, x'\in X$ be arbitrary elements of $X$ and let $f : X \rightarrow A$ be a $1$-cocycle. By indecomposability  of $X$ there exist $y_1, \ldots , y_n\in X$ such that $(\cdots (x*^{\epsilon_1}y_1)*^{\epsilon_2}\cdots *^{\epsilon_n}y_n )= x'$, such that $\epsilon_i=\pm 1$, where $*^{-1}$ is defined by $x*^{-1}y=z$ if %and only if
$z*y=x$.
Recall that in this case the 1-cocycle condition is $f(x*y)=f(x)$, which also implies $f(x*^{-1} y)=f(x)$.
Therefore
\[f(x') = f(\cdots (x*^{\epsilon_1}y_1)*^{\epsilon_2} \cdots *^{\epsilon_n} y_n ) = f(\cdots (x*^{\epsilon_1}y_1)*^{\epsilon_2} \cdots *^{\epsilon_{n-1}}y_{n-1} )\]
where the second equality follows from the $1$-cocycle condition for $f$.
Inductively
it follows that $f$ is a constant map.
On the other hand, any constant map satisfies the cocycle condition and is continuous, hence it is in $H^1_C(X,A)=Z^1_C(X, A)$.
Hence there is a bijective correspondence between $H^1_C(X,A)$ and $A$
that
respects the group structures as the group operation of cocycles is pointwise, and
the lemma follows.
\end{proof}

\begin{pro}
Let $X=({\mathbb R}^n, S)$ and $A=({\mathbb R}^m, T)$
be  indecomposable Alexander quandles, where $S, T$ are continuous additive automorphisms.
Then $H^1_{TC} (X, A)$ is isomorphic to
\[
\{ \ F + a : {\mathbb R}^n  \rightarrow {\mathbb R}^m \ | \ a \in A, \ F \ \text{is linear,}\  FS=TF \ \} .
\]
\end{pro}

\begin{proof}
Since $X$ and $A$ are indecomposable, we have $I-S$ and $I-T$ invertible.
Let  $G \in H^1_{TC} (X, A)$.
Then $G$ is a continuous quandle homomorphism $G: X \rightarrow A$.
Then for all $a \in A$, $G+a \in H^1_{TC} (X, A)$.
For any $G \in H^1_{TC} (X, A)$, there is $a \in A$ such that $(G+a)(0)=0$.
By Lemma~\ref{lem:similar}, $F=G+a$ is linear, and $FS=TF$.
Hence the result follows.
\end{proof}

\begin{pro}
Let $X=({\mathbb R}^n, T)$
be  an indecomposable Alexander quandle.
Then $H^1_{\rm TC}(X, X) \not\cong H^1_{\rm T}(X, X)$.
\end{pro}

\begin{proof}
By Proposition~\ref{prop:noniso}, there are quandle isomorphisms that are not continuous.
Hence we have that $H^1_{\rm TC}(X, X)$ is a proper subgroup of $H^1_{\rm T}(X, X)$.
\end{proof}

\subsection{Continuous Second Cohomology and Extensions}

We  define the concept of extension of a quandle, in the topological context, following \cite{CENS, CES}, see also \cite{EM}.

Assume we are given a quandle $X$ and an Alexander quandle
 $(A,T)$.
For a $2$-cocycle $\psi \in Z^2_{T}(X, A)$ (so that $\; \psi(x,x)=0 $ for all $x \in X$), a quandle structure is defined on $X \times A$ by
\[ (x,a)*(y,b) = (x*y, a*b+ \psi (x,y))\]
 for all $x,y \in X$ and $a,b \in A$, as in  \cite{CENS}.
 The resulting quandle is denoted by
 $X \times_{\psi} A$ and called an {\it extension} of $X$ by $A$.
 The projection $\pi: X \times_{\psi} A \rightarrow X$ is a quandle homomorphism.

For a topological quandle $X$ and a topological Alexander quandle $(A, T)$, we observe the following.
Assume $\psi \in Z^2_{TC}(X, A)$, a continuous 2-cocycle.
Then $X \times_{\psi} A$ has the product topology and  the quandle operation on $X \times_{\psi} A$
is continuous, therefore $X \times_{\psi} A$ is a topological quandle.
Furthermore, the projection $\pi: X \times_{\psi} A \rightarrow X$ is continuous, hence
$\pi$ is a topological quandle morphism.
Although we call $X \times_{\psi} A$  the extension of
a topological quandle, we require continuity as described above.
For topological spaces, it appears desirable to consider fiber bundles instead of product spaces. However, for the purpose of investigating differences of cohomology groups between continuous and algebraic, and for non-triviality of continuous cohomology groups, in this article we focus on the product cases. See also Section~\ref{sec:principal} that is related to this issue.

 We
 define morphisms in the class of %abelian
 extensions of $X$ by the abelian group $A$ and, consequently, define an equivalence relation corresponding to the isomorphism classes. The class of  extensions of $X$ by $A$ can be viewed therefore as a category.
 Consider two topological  extensions $X \times_{\psi} A$ and $X \times_{\phi} A$
 where $\psi$ and $\phi$ are two $2$-cocycles. A morphism $X \times_{\psi} A  \rightarrow X \times_{\phi} A$ of  extensions of $X$ by $A$ is a morphism of topological quandles $f: X\times_{\psi}  A \rightarrow X\times_{\phi} A$  making the following diagram commute:
 \[
 \begin{tikzcd}
 X\times_\psi A \arrow [r, "f"]\arrow[d] &  X\times_\phi A\arrow[d] \\
 X \arrow[r, equal] & X
 \end{tikzcd}
 \]
  In particular, if $f$ is an isomorphism of topological quandles with the property of making the above diagram commute, it will be called an {\it isomorphism of topological  extensions}.
  Two extensions are {\it equivalent} if there is an isomorphism $f$ as above.
 We now prove the following result, analogous to the classification of the second cohomology group for group cohomology and the corresponding result for discrete quandles, as in \cite{CENS, CES}.

 \begin{pro}
 There is a bijective correspondence between equivalence classes of topological abelian extensions of $X$ by $A$ and the second cohomology group $H_{TC}^2(X,A)$ of $X$ with coefficients in $A$.
 \end{pro}

\begin{proof}
Although computations below are similar to those in \cite{CES}, we examine   topological aspects of the argument.
 Assume $ X\times_\psi A $ and $ X\times_\phi A $ are two topological  extension of $X$ with $\psi$ and $\phi$ cohomologus $2$-cocycles (i.e. they differ by a coboundary). Consider the map $f: X\times A \rightarrow X\times A$, $ (x,a) \mapsto (x, a + g(x))$, where $g : X \rightarrow A$ is such that $\delta g = \psi - \phi$.
 Since $g \in Z^1_{TC}(X,A)$, $g$ is continuous, and so is $f$.
 We have
 \[ f((x,a)*(y,b)) = f(x*y, a*b +\phi(x,y)) = (x*y, a*b  + \phi(x,y) + g(x*y)).\]
 On the other hand we have
 \[ f(x,a)*f(y,b) = (x, a+ g(x))*(y, b + g(y)) = (x*y, a*b+ g(x)*g(y) + \psi(x,y)).\]
  These two terms are equal since $\phi = \psi + \delta g$, hence $f$ is an isomorphism of quandles.
 Since it is also a homeomorphism and clearly makes the required diagram commute, we get that $X\times_\psi A $ and $X\times_\phi A $ are equivalent.

 Conversely, assume $X\times_\psi A $ and $X\times_\phi A $ are equivalent. Say $f: X\times A \rightarrow X\times A$ is an isomorphism of topological  extensions. Since, by definition, both $\pi (x,a)=x$ and $\pi (f(x,a))=x$, % lie over $x\in X$,
 the map $f$ is determined by its second component. Using the group structure of $A$ we can also write $f$ as
 $f(x,a)=(x, a+g(x))$
 for some map $g : X \rightarrow A$.  Now the continuity of $f$ implies the  continuity of $g$.
Since $f$ is a morphism of quandles we get, for all $x,y \in X$ and all $a,b\in A$,
\begin{eqnarray*}
\lefteqn{
  (x*y, a*b + g(x*y) + \psi(x,y)) = f((x,a)*(y,b)) }\\
 &=&
f(x,a)*f(y,b) = (x*y, a *b + g(x)*g(y) + \phi(x,y)).
\end{eqnarray*}
Equating  the second component, we find that $\psi$ and $\phi$ differ by $\delta g$,  i.e. they are representatives of the same cohomology class, since $g$ is continuous.
\end{proof}

\begin{lem}\label{lem:Kronecker}
Let $X$ be a topological quandle, and $(A, T)$ be a topological Alexander quandle.
Let $\alpha \in Z_n^{\rm T}(X, A)$ be an $n$-cycle (in the usual sense of discrete homology), and $\phi \in Z^n_{\rm TC}(X, A)$ be a continuous $n$-cocycle.
If $\phi(\alpha )\neq 0$, then $[\phi] \neq 0 \in H^n_{\rm TC}(X, A)$.
\end{lem}

\begin{proof}
The standard argument applies as follows in continuous case.
Suppose $\phi$ is a coboundary in $H^2_C(X,A)$. Then, by definition, there exists some continuous $f: X \rightarrow A$ such that $\delta f = \phi$ which, in particular, means that $[\phi] = 0$ in $H^2(X,A)$ (discrete cohomology).
As a consequence, applying the Kronecker pairing we get:
\[ \phi(\alpha ) = \bra{[\phi]}\ket{[\alpha]} =
 \bra{[\delta f ]}\ket{[\alpha]} =  \bra{[ f ]}\ket{[\partial \alpha]}  = 0\]
 since $\alpha \in Z_n^{\rm T}(X, A)$.
This contradicts the assumption.
\end{proof}

	Let $(G,  +)$ be a topological abelian group.
 Consider $G^m$ with the quandle structure given by the rule:
\[
(a_1, \ldots , a_m) * (b_1, \ldots , b_m) = (a_1, a_2 + b_1 - a_1, \ldots , a_m + b_{m-1} - a_{m-1}).
\]
By direct computation we see that the operation just defined respects the defining axioms for a quandle structure and it is continuous, hence define a topological quandle structure on $G^{m}$.
This construction is motivated from \cite{CENS}.

\begin{pro}
Let $(G, +)$ be a topological abelian group, $ x \neq 0$,  and
$(G^m, *)$ be the topological quandle defined as above.
Then $H^2_{\rm C}( G^{m}, G ) \neq 0$.
\end{pro}

\begin{proof}
Consider the following $2$-cycle (in the usual sense of discrete homology):
\[
\alpha = (0, \ldots , 0) \times (x, 0, \ldots , 0) + (0, x , 0, \ldots , 0)\times (-x, x, 0, \ldots , 1) \ ,
\]
where $\times$ has been used to better indicate that $\alpha$ is an element of $G^{m}\times G^{m}$.
By direct computation using the boundary map, it follows that $\alpha$ is indeed a $2$-cycle.
Consider also the $2$-cocycle:
\[
\phi : G^{m}\times G^{m} \longrightarrow G \]
defined by
\[ \phi ( \, (a_1, \ldots , a_m)\times(b_1, \ldots , b_m) \, ) = b_m - a_m \ .
 \]
Again by direct computation using the coboundary map it can be shown that $\phi$ is a cocycle.
Applying $\phi$ to $\alpha$ we get $\phi (\alpha) = x \neq 0$. Hence
by Lemma~\ref{lem:Kronecker} $\phi$ is not null-cohomologus, and we obtain $H^2_{\rm C}( G^{m}, G )\neq 0$.
\end{proof}

Note that a constant map $\phi$ in $\Gamma^n(X,A)$
is the zero map for $n>1$.

\begin{pro}\label{prop:linear}
Let $X=({\mathbb R}^n, S)$ and $A=({\mathbb R}^m, T)$ be  Alexander quandles, where $S \in {\rm GL}_n(\mathbb{R})$ and
 $T \in {\rm GL}_m(\mathbb{R})$, respectively.
Then $H^2_{\rm TC} ( X, A ) \neq 0$ if one of the following conditions hold.
\begin{itemize}
\item[{\rm (i)}]
There is a nonzero vector $v$ and $m \times n$ matrix $C$ such that
$  S^2-S+1 =  T^2 - T +1  =0 $ and $ C - TCS \neq  0$.
\item[{\rm (ii)}]
$T=-I $ and $S$ has eigenvalue $-1$.
\end{itemize}
\end{pro}

\begin{proof}
	Consider the matrix $M:=\begin{pmatrix}
	S & 0 \\
	C & T
	\end{pmatrix} \in {\rm GL}_{m+n}(\mathbb{R})$.
	Then $X \times A$
	can be considered as an Alexander quandle with operation
	$$\begin{pmatrix}
		x_1\\
		a_1
	\end{pmatrix} *\begin{pmatrix}
	x_2\\
	a_2
	\end{pmatrix}:= M \begin{pmatrix}
	x_1 \\
	a_1
	\end{pmatrix} +(I_{m+n} -M)\begin{pmatrix}
	x_2 \\
	a_2
	\end{pmatrix} . $$
	 A direct computation shows that $\begin{pmatrix}
	x_1 \\
	a_1
	\end{pmatrix} *\begin{pmatrix}
	x_2 \\
	a_2
	\end{pmatrix}= \begin{pmatrix}
	x_1*x_2\\
	a_1*a_2 +\phi(x_1,x_2)
	\end{pmatrix} $ and thus the quandle $X \times A$ is a topological twisted extension of the topological quandle $X$ with the continuous $2$-cocycle $\phi: \mathbb{R}^n \times  \mathbb{R}^n \rightarrow \mathbb{R}^m$ given explicitly by the formula
	\[
	\phi(x_1,x_2)=C(x_1-x_2).
	\]
We find a 2-cycle $w \in Z_2^{\rm T}(X,A)$ such that $\phi(w)\neq 0$ under each of the conditions (i) and (ii), then
Lemma~\ref{lem:Kronecker} implies $H^2_{\rm TC} ( X, A ) \neq 0$.
Let $w=(u_0, v_0) + T(u_1, v_1)$. This is an element of the free module over $A$ generated by pairs of elements of
$X$  as in \cite{CES}.
Then one computes
\begin{eqnarray*}
\partial w &=&
[T (u_0) + (I-T) (v_0) - (S u_1 + (I-S) v_0) ) ] \\
& & + \quad T [ T (u_1) + (I-T) (v_1) - (S u_1 + (I-S) v_1) ) \\
&=&
[ ( v_0)  -  (S u_1 + (I-S) v_0)  ] \\
& & +
T [ (u_0) - ( v_0 ) + ( v_1 ) - (S u_1 + (I-S) v_1) ]
+
T^2 [ (u_1) - ( v_1) ]
. \quad (*)
\end{eqnarray*}
where $(v_1)$ is a 1-chain, and not an element of $X$.

(i) Set $u_1=v_0$ and $v_1=S u_0  + (I- S) v_0$. Then the expression  $(*)$ vanishes if
$$(S^2 - S + 1) (u_0 - v_0)=(T^2 - T + 1)(u_0 - v_0) =0 . $$
Hence the assumption $  S^2-S+1  = 0= T^2 - T +1  $ implies  $\partial w=0$.

One computes $\phi(w)= ( C - TCS) (u_0 - v_0)$.
One can choose  nonzero vectors $u_0$ and $v_0$ satisfying this condition under the assumption  $C - TCS \neq 0$.
Hence
$\phi(w)= ( C - TCS) (u_0 - v_0) \neq 0 $ follows from the assumption.
From Lemma~\ref{lem:Kronecker} we obtain that $\phi$ is not null cohomologous.

(ii) Set  $T=-I$. Further, set $v_0=0$, $u_1= -u_0$, and $v_1=0$. Then $(*)$ vanishes if there exists
$u_0$ such that $Su_0=-u_0$, and this condition is satisfied by the assumption.
Then one computes
$\phi(w)= C(u_0 - 0) + C( - u_0 - 0) = C( 2 u_0)$. By choosing $C$ and $u_0$ such that $C u_0 \neq 0$,
one obtains a non-trivial 2-cocycle $\phi$.
\end{proof}

\begin{ex}
	Let $n=4$, $m=2$, $S=T \oplus T$, and $T=\begin{pmatrix}
	0 & -1\\
	1 & 1
	\end{pmatrix}$.
	Then $S^2-S+1=0=T^2-T+1$.
	Let $C=(I, I)$ where $I$ is $2\times 2$ identity matrix. Then $C-TCS= (I-S^2 , I-S^2 )$
	is not the zero matrix,
	and conditions in (i) in  Proposition~\ref{prop:linear} are satisfied, and we obtain  $H^2_{\rm TC} ( X, A ) \neq 0$
	for this specific example.
\end{ex}

\begin{rmk}
	The conditions of Proposition~\ref{prop:linear} (i) can be extended to those for higher degree alternating polynomials.
	Since the idea of proof is the same but the computations are lengthy, we delay its statement and proof to Appendix~\ref{sec:nontriv}.
\end{rmk}

\subsection{Continuous Cohomology with Quandle Modules}\label{sec:AG}

The goal of this subsection is to introduce  a topological version of the cohomology theory generalized in
\cite{AG}  and exhibit explicit examples with non-trivial continuous cohomology.
Before defining the generalized cohomology theory of topological quandles we
present a few
preliminary concepts.

 Let $X$ be a topological quandle.  Recall~\cite{AG, EM} that an $X$-module is a triple $(A,\eta,\tau)$ where $A$ is a topological abelian group, $\eta$ is a family of continuous group automorphisms $\eta_{x,y}:A\rightarrow  A$  and $\tau$  is a family of continuous group morphisms $\tau_{x,y}:A \rightarrow  A$
 such that

\begin{enumerate}
\item
$\eta_{x * y,z}\; \eta_{x,y} =  \eta_{x * z,y * z}\; \eta_{x,z}$,
\item
$ \eta_{x * y,z}\; \tau_{x,y} =  \tau_{x * z,y * z}\; \eta_{y,z}$,
\item
$\tau_{x * y,z}= \eta_{x * z,y * z} \tau_{x,z}+
 \tau_{x * z,y * z}\tau_{y,z}$, \quad and

  \item
   $\tau_{x,x} + \eta_{x,x} ={\rm id}_{ A}$.
\end{enumerate}

We consider the abelian groups
 $\Gamma^n(X, A)$, $\delta^i_0$ and $\delta^i_1$  as defined above in Section~\ref{secCohomology}.  Define the differentials by the following formula
  \begin{eqnarray*}
 \delta^n & : =&\sum_{i=2}^{n+1}(-1)^i\left(\eta_{[x_1,\cdots,\widehat{x}_i,\cdots, x_{n+1}],[x_i,\cdots, x_{n+1}]}\delta^i_0 - \delta^i_{  1}\right) \\
  & & + \ \tau_{ [x_2, x_3, \ldots, x_{n+1}], [x_1, x_3, \ldots, x_{n+1}]}\delta^1_0
 \end{eqnarray*}
 where
 $$[x_1, x_2, x_3, \ldots, x_{n}] = (( \cdots ( x_1 * x_2) * x_3 ) \cdots ) * x_n . $$
  The resulting cohomology groups  are denoted by
 $H^n_{GC}(X, A)$.

 As in the discrete case, it is easy to see that we obtain a cochain complex \[\cdots \rightarrow \Gamma^n(X,A) \xrightarrow{\delta} \Gamma^{n+1}(X,A)\rightarrow\cdots\]
  Following \cite{AG}, if $X$ is a topological quandle and  $(A, \eta, \tau)$ is a topological quandle module, we can define a topological quandle structure on $X \times A$ by:
\[ (x,a)*(y,b) = (x*y, \eta_{x,y}(a) + \tau_{x,y}(b) + \kappa_{x,y}),\]
 for all $x,y \in X$ and $a,b \in A$, where $X \times A$ is given the product topology.
This formula defines a topological quandle structure if and only if
$\kappa_{x,y}$ is a $2$-cocycle of this cohomology theory.

  We have that the
  projection onto the
  first factor
   is
   an epimorphism of topological quandles.

Let $G$ be the subgroup of ${\rm GL}(2n, \mathbb{R})$ for a positive integer $n$, consisting of block matrices of the form
$$G= \left\{ \
\left. E= \left(
\begin{array}{cc} S & O \\ C & T \end{array}
\right) \
  \right| \
S, T  \in {\rm GL}(n, \mathbb{R} ), C \in {\rm M}(n, \mathbb{R} ) \
\right\} , $$
where $O$ denotes the zero matrix,
and consider $X=G  \times  \mathbb{R}^n $
with quandle operation
$$
( E_0, x_0) * ( E_1,   x_1)=
( E_1 E_0 E_1^{-1},   S_1 x_0 + (I - S_1) x_1 )  ,
$$
where
$E_i = \left(
\begin{array}{cc} S_i & O \\ C_i & T_i \end{array}
\right)
$
for $i=0,1$.
Let $A=\R^{n}$ and consider endomorphisms of $A$ defined by
$\eta_{ (E_0, x_0), (E_1, x_1) } (a) = T_1 a $ and  $\tau_{ (E_0, x_0), (E_1, x_1) } (a) = (I - T_1) a $.
It is checked by computation that these define a $X$-module structure on $A$.

\begin{thm}\label{thm:AG}
$H^2_{GC}(X, A) \neq 0$.
\end{thm}

\begin{proof}
The quandle operation on $X \times A= G \times \mathbb{R}^{2n} $ defined by
$$
( E_0, u_0) * ( E_1,   u_1)=
( E_1 E_0 E_1^{-1},   E_1 u_0 + (I - E_1) u_1 )  ,
$$
where $u_i=(x_i, a_i)$ and $a_i \in A$ for $i=0,1$ are computed as operation on $X \times A$ as
\begin{eqnarray*}
\lefteqn{[ \   (E_0, x_0), a_0 \ ] * [ \ ( E_1,  x_1) , a_1 \ ] }\\
& = &
[ \  E_1 E_0 E_1^{-1},   S_1 x_0 + (I - S_1) x_1,  T_1a_0 + (I - T_1) a_1  +  C_1 (x_0 - x_1) ) ] .
\end{eqnarray*}
Let $p: X \times A \rightarrow X$ by $p(\ [ (E, x) ,a) ] \ ) = (E, x)$.
Then we find that $p$ defines the extension of $X$ by the $X$-module $A$,
with the 2-cocycle
$\kappa_{( E_0, x_0), ( E_1, x_1)}=  C_1 (x_0 - x_1)$.

We show that $\kappa$ is not a coboundary.
Let $E= \left(
\begin{array}{cc} -I  & O \\ C & -I  \end{array}
\right) $, and
 $w= [ (E, x), (E, 0) ] - [ ( E, -x), (E, 0) ]$ be a 2-chain.
 Since
 $\partial( \, ( x, y) \, ) = \eta_{x, y} (x ) + \tau_{ x, y}(y) - (x*y)$,
one computes
  \begin{eqnarray*}
\partial (w) & = &  (-I) ( E,  x)  + (2I) ( E, 0 ) - (E, (-I)x  +(2I) 0  ) \\
& - & [   (-I) ( E,  - x)  + (2I) ( E, 0 ) - (E, (-I)(-x)  +(2I) 0  ) \quad = \quad 0.
 \end{eqnarray*}
 Hence $w$ is a 2-cycle.
One also computes
\begin{eqnarray*}
 \kappa_{( E, x) , (  E, 0) } - \kappa_{  ( E, -x) , (E, 0) }
&=&
 C_1(x - 0 ) - C_1(-x - 0 )  \\
&=&
 C_1(2x )
\end{eqnarray*}
and by choosing $x, C_1$ such that $C_1x \neq 0$, we obtain that $\kappa$ is not a coboundary
by the argument similar to Lemma~\ref{lem:Kronecker}.
\end{proof}

The construction of $X\times A$ in Theorem~\ref{thm:AG} is similar to the notion of $G$-family of quandles defined in \cite{IIJO}.

\subsection{Principal Bundles as Quandle Extensions}
\label{sec:principal}

In this section we describe  quandle extensions of topological quandles
that are not defined by continuous 2-cocycles but can be described by principal bundles in a natural manner with discontinuous 2-cocycles.
As a consequence we observe the difference between continuous and algebraic cohomology theories.

Let $A$ be a topological abelian group and $p: E \rightarrow X$ be a principal $A$-bundle;
a fiber bundle with  a fiber preserving right action of $A$ on $E$ that acts freely and transitively.

\begin{df}[cf.~\cite{Eis-cover}]
{\rm
Let $E,X$ be connected topological quandles and  $A$ be a topological abelian
group.
A {\it principal  (abelian)
quandle extension by $A$} is a continuous surjective quandle homomorphism
$p: E \rightarrow X$ that is a principal $A$-bundle
such that
for all $x, y \in X$ and $a \in A$, the following conditions hold:
(i)  $(x*y)\cdot a= (x \cdot a)* (y \cdot a)$, % (equivariance),
(ii) $ (x \cdot a)*y= (x * y) \cdot a$. % (commutativity of right actions).
}
\end{df}

The quandle homomorphism in Example~\ref{ex:Rubin} is a  principal  abelian quandle extension by $\Z_2$.

\begin{lem}
Let $A$ be a topological abelian
group and
$p: E \rightarrow X$ be a  principal abelian quandle extension by $A$.
Let $s: X \rightarrow E$ be a set-theoretic section; $p \circ s={\rm id}_X$.
Then for all $x, y \in X$, there exists a unique element $a \in A$ such that
$s(x)*s(y)=s(x*y)\cdot a$.
\end{lem}

\begin{proof}
Since $p$ is a quandle homomorphism, we have
$$p( s(x)*s(y) ) = (p s)(x)* ( ps)(y)= x*y= (ps)(x*y). $$

Since $A$ acts  freely and transitively, there is a unique $a$ such that $s(x)*s(y)=s(x*y)\cdot a$.
\end{proof}

In the preceding lemma, the unique element $a$ is determined by $x, y \in X$, so that we denote it by $a=\phi(x,y)$.
Then  we obtain a function $\phi: X \times X \rightarrow A$.

\begin{lem} \label{lem:princ2cocy}
Let $\phi: X \times X \rightarrow A$ be  defined as above.
Then $\phi$ is a quandle (abelian) 2-cocycle.

\end{lem}

\begin{proof}
We perform the following computations analogous to those in \cite{CENS} and \cite{Eis-cover}:
\begin{eqnarray*}
 (s (x) * s(y) ) * s(z)
 &=& [\, s( x*y ) \cdot \phi(x, y) \, ] * s(z) \\
 &=& [\, s(x*y) * s(z) \, ] \cdot \phi (x,y) \\
 &=&  s( (x*y)*z )
  \cdot [ \phi (x*y, z) \phi (x,y) ] , \\
  (s (x) * s(z) ) * ( s(y) * s(z) )
  &=&  [\, s( x*z ) \cdot \phi(x, z) \, ] *  [\, s( y*z) \cdot \phi(y, z) \, ] \\
   &=& ( [\, s( x*z ) \cdot \phi(x, z)\phi(y, z)^{-1}  \, ] *  s( y*z) ) \cdot \phi(y, z)  \\
      &=& ( s( x*z )  *  s( y*z) ) \cdot [  ( \phi(x, z)\phi(y, z)^{-1}  ) \phi(y, z)  ] \\
   &=& s( ( x*z )  * ( y*z) ) \cdot [  (\phi(x*z, y*z)  \phi(x, z)  ] ,
 \end{eqnarray*}
 and $s(x)*s(x)=s(x*x)\cdot \phi(x,x)$ gives $\phi(x,x)=0$.
Hence $\phi$ satisfies the 2-cocycle condition.
\end{proof}

Nosaka pointed out to us that the argument works also for non-abelian groups $A$.
See \cite{AG} for non-abelian 2-cocycles.

\begin{ex}\label{ex:proj}
{\rm
Consider $p: S^2 \rightarrow \mathbb{RP}^2$ in Example~\ref{ex:Rubin}.
Let
$$P_+:= \{ (x,y,z) \in S^2 : z >0 \ {\rm or} \
z=0, y>0 \ {\rm or} \ y=z=0, x>0 \}$$
 and $P_-:= S^2 \setminus P_+$.
Let $s: \mathbb{RP}^2 \rightarrow S^2$ be a set-theoretic section defined by
$s( [x]) =x$ where $x\in P_+$.
Then the map $\phi$ of the preceding lemma provides a non-zero 2-cocycle.
For example, $\phi( [1,0, 0] , [0,1,0]) = 1 \in \Z_2$.
In this case, as a set $S^2$ is regarded as $\mathbb{RP}^2 \times \mathbb{Z}_2$.
}
\end{ex}

\begin{rmk}
{\rm
Let  $p: E \rightarrow X$ be a  principal abelian quandle extension by $A$,
and fix a set-theoretic section $s: X \rightarrow E$.
For any given  $u \in E$, let $x=p(u)$, then there is a unique $a=a_s(u)$ such that $u=s(x)\cdot a$.
Similarly for $v \in E$ let $y=p(v)$ and $v=s(y) \cdot b$.
Then one computes
$$ u*v= ( s(x) \cdot a ) * ( s(y) \cdot b ) = [ (  s(x) \cdot (a b^{-1} ) ) * s(y) ] \cdot b = [ s(x) * s(y) ] \cdot a
=  s(x*y)  \cdot  ( a \/ \phi(x, y) ) . $$
Note that this equality  $ ( s(x) \cdot a ) * ( s(y) \cdot b ) =  s(x*y)  \cdot  ( a \, \phi(x, y) )$
compares to the equality
$ (x, a) * (y, b)= (x*y, a + \phi(x, y) ) $ for $E=X \times_\phi A$ in the case $T=1$.
}
\end{rmk}

\begin{lem} \label{lem:zero}
Let $p: E \rightarrow X$ be a  principal abelian quandle extension by $A$.
If $A$ is discrete and $\phi$ is a continuous 2-cocycle, then $\phi(x,y)=0$ for all $x, y \in X$.
\end{lem}

\begin{proof}
From the assumptions $\phi$ is a constant map. From the quandle condition $\phi(x,x)=0$ for all $x \in X$,
it follows that $\phi(x,y)=0$ for all $x, y \in X$.
\end{proof}

Note that in this case  it follows that $s$ is a quandle homomorphism.

For  $p: S^2 \rightarrow \mathbb{RP}^2$ in Example~\ref{ex:Rubin} we have the following proposition.

\begin{pro}
$H^2_{\rm Q}(\mathbb{RP}^2, \Z_2) \neq 0$, yet   $H^2_{C} (\mathbb{RP}^2 , \Z_2)=0$.
\end{pro}

\begin{proof}
Let $\phi$ be the  quandle 2-cocycle constructed in Example~\ref{ex:proj}.
By Lemma~\ref{lem:princ2cocy} and Example~\ref{ex:proj}, $\phi$ is a (discrete) quandle $2$-cocycle,
that yields a non-trivial extension, and therefore, $\phi$ is non-trivial in $H^2_{\rm Q}(\mathbb{RP}^2, \Z_2) $.

By Lemma~\ref{lem:zero}, any continuous 2-cocycle gives rise to the trivial extension, and hence
$H^2_{C} (\mathbb{RP}^2 , \Z_2)=0$.
\end{proof}

\section{Inverse Limits of Quandles and their Cohomologies}\label{InvLim}

In this section we apply a method of computing cohomology of inverse limits to determine continuous cohomology of quandles.

\subsection{Basic Construction of Inverse and Direct Limits}

Suppose we are given a projective system of quandles $(X_n, \psi_n)_{n\in\mathbb{N}}$:
\[
X_1 \stackrel{\psi_1}{\longleftarrow} X_2 \stackrel{\psi_2}{\longleftarrow}  \cdots \longleftarrow X_n \stackrel{\psi_n}{\longleftarrow} \cdots
\]
where each $\psi_n$ is a quandle morphism.
We
define the inverse limit of the projective system, $\varprojlim X_n$, as the subset of $\prod_{n\in\mathbb{N}} X_n$ of sequences $(x_0, x_1, \ldots , x_n, \ldots)$ satisfying $\psi_n(x_{n+1}) = x_{n}$ for all $n\geq 1$.
We
give $\varprojlim X_n$ the $*$ operation induced componentwise by the operations of the $X_n$. This construction together with the canonical projection maps $\varprojlim X_n \rightarrow X_i$, for each $i \in \mathbb{N}$, satisfies the usual universal property for an inverse limit of a projective system indexed by the natural numbers, see Remark~\ref{rem:invsys} below.

The same construction can be defined
 for a projective system of topological quandles, where each morphism is now required to be continuous, and $\varprojlim X_n$ is endowed with the initial topology with respect to the projection maps.
 The initial topology is the coarsest topology that makes projections $p_i : \varprojlim X_n \rightarrow X_i$ continuous, and in our case,
 it is the same topology as the subspace topology of the product space.

A particular case is to start with a projective system of discrete (possibly finite) quandles. The topology on $\varprojlim X_n$ will be the subspace topology of the product $\prod_{n\in\mathbb{N}} X_n$, where the latter space has the product topology of discrete spaces.
Thus taking the inverse limit provides a method to build topological quandles from discrete quandles.
The following example is of central importance.

\begin{ex}
Fix a prime $p\in \mathbb{N}$. Put $X_n = \mathbb{Z}/p^n\mathbb{Z}$ together with the standard dihedral quandle operation $x*y=2y-x$. There are canonical projections $\psi_n : X_{n+1} \rightarrow X_{n}	$ obtained by reducing $mod \ p^{n}$ a representative of a class modulo $p^{n+1}$. These maps are %easily seen to be
ring homomorphisms and, as a consequence, quandle morphisms, since the quandle structure on $X_n$ is obtained from the ring operations.
By definition, the inverse limit of this projective system is the ring of p\,-adic integers $\mathbb{Z}_p$ and it inherits a topological quandle structure from the dihedral quandles $X_n$.
To be precise, the same quandle operation would be obtained on $\mathbb{Z}_p$ defining the Alexander quandle structure with $T = -1$, via the ring operations on $\mathbb{Z}_p$. That is,
$( \varprojlim X_n, *)$ is isomorphic to $( \mathbb{Z}_p , -1)$.
\end{ex}

\begin{rmk}\label{rem:invsys}
More generally, if we start with a directed set $I$ and a projective system of (topological) quandles $(X_i, \psi_{i,j})$, where the morphisms $\psi_{i,j}$ satisfy the usual compatibility relations, we can define $\varprojlim X_n = \{ x\in\prod_{n\in I} X_n\ | \ \psi_{i,j}\pi_j(x) = \pi_i(x) \ $for all$ \ i,j \ $ with$ \  j \geq i  \}$  where $\pi_i$ is the canonical projection onto the $i^{th}$ factor.
They satisfy the universal property depicted below.

\[	\begin{tikzcd}
		& Y\arrow[d,dashed,"\exists!"]\arrow[ddr] \arrow[ddl]& \\
		& \varprojlim X_n\arrow[dr]\arrow[dl] &     \\
X_j\arrow[rr,"\psi_{i, j}"]	&         & X_i
			\end{tikzcd}\]

Then we can  endow it again with the $*$ operation induced pointwise by the quandle operations in each $X_i$ and get an inverse limit for the projective system we started with.
 If we start with topological quandles, the topology of $\varprojlim X_i$ will be again the initial topology with respect to the projections.
By definition it follows that the inverse limit of (topological) quandles, is the usual inverse limit in the category of topological spaces, equipped with a continuous binary operation that turns it into a topological quandle.
\end{rmk}

We
recall the following results from point-set topology, proofs of which  can be found in \cite{Wilson}:
\begin{lem}
	Let $(X_i, \psi_{i,j}) $be a projective system of topological spaces. Then the following results hold:
	\begin{itemize}
	\item[{\rm (a)}] If each $X_i$ is compact and Hausdorff, so is $\varprojlim X_i$;
	\item[{\rm (b)}]  If each $X_i$ is totally disconnected, so is $\varprojlim X_i$;
	\item[{\rm (c)}]  If each $X_i$ is a nonempty compact Hausdorff, then $\varprojlim X_i$ is nonempty.
	\end{itemize}
\end{lem}

\begin{lem}\label{lem:cptH}
	Let $(X_i, \psi_{i,j})$ be a projective system of nonempty compact Hausdorff spaces and let $Y$ be a discrete space.
	Then any continuous map $f : \varprojlim X_i\rightarrow Y $ factors through $X_i$ for some $i \in I$.
\end{lem}

Observe also that taking inverse limits in the category of topological spaces commutes with the operation of taking cartesian products, $ (\varprojlim X_i)^{\bullet} = \varprojlim (X_i^{\bullet})$.

A similar construction applies to
 an inductive system of quandles
\[
X_1 \xrightarrow{\phi_1} X_2 \xrightarrow{\phi_2} \cdots \xrightarrow{\phi_{k-1}} X_k \xrightarrow{\phi_k} \cdots
\]
to obtain a quandle structure on the direct limit
$\varinjlim X_k$. In this case $\phi_k$ are continuous quandle morphisms.
The direct limit is endowed with  the final topology (or inductive topology),  the finest topology which makes the  functions $X_k \rightarrow \varinjlim X $ continuous.
In particular, if $X_k$ are discrete, then $\varinjlim X_k$
is discrete.

\subsection{Cohomology of Inverse Limits}
The purpose of this section is to apply the tools from the continuous cohomology groups of the inverse limit of a projective system to finite discrete quandles $(X_n, \psi_n)_{n\in\mathbb{N}}$, to obtain the cohomology groups of $\varprojlim X_n$.
The tools used here depend only on the functoriality of twisted cohomology of quandles, the universal property of colimits and the factorization property of continuous maps between an inverse limit and a discrete space.
Suppose we are given a projective system $(X_n, \psi_n)$ of discrete quandles:
\[
X_1 \xleftarrow{\psi_1} X_2 \xleftarrow{\psi_2}\cdots \xleftarrow{\psi_{n-1}}X_n \xleftarrow{\psi_n} \cdots
\]
We also consider an inductive system of Alexander quandles $(A_k,\phi_k)$, where $(A_k, T_k)$ are topological Alexander quandles ($A_k$ is a topological abelian groups and $T_k$ are continuous automorphisms of $A_k$):
\[
A_1 \xrightarrow{\phi_1} A_2 \xrightarrow{\phi_2} \cdots \xrightarrow{\phi_{k-1}} A_k \xrightarrow{\phi_k} \cdots
\]
From the data above we can construct an inductive system of cohomology groups.
Below we suppress subscripts of cochain groups for simplicity.
Define cochain maps $C^\bullet(X_n, A_n) \longrightarrow C^\bullet(X_{n+1}, A_{n+1})$ by the following diagram:
\[
\begin{tikzcd}
C^\bullet(X_n, A_n) \arrow[r,"\psi_n^*"]\arrow[dr,dashed,"\tau_n"] & C^\bullet(X_{n+1}, A_{n})\arrow[d,"\phi_n \circ -"]\\
 & C^\bullet(X_{n+1}, A_{n+1})
\end{tikzcd}
\]
where the vertical map is the change of coefficients induced by $\phi_n$ and the horizontal map is the dual map of the projection $\psi_n$.
We obtain consequently an inductive system in the category of groups.
Consider now the following diagram:
\[\begin{tikzcd}
& \vdots \arrow[d,"\delta^{\bullet-1}"]& \vdots \arrow[d,"\delta^{\bullet-1}"]&\\
\cdots \arrow[r] & C^\bullet(X_n,A_n) \arrow[r,"\tau_n"]\arrow[d,"\delta^\bullet"] & C^\bullet(X_{n+1},A_{n+1}) \arrow[r]\arrow[d,"\delta^\bullet"] & \cdots \\
\cdots \arrow[r]& C^{\bullet+1}(X_n,A_n) \arrow[r,"\tau_n"]\arrow[d,"\delta^{\bullet+1}"]& C^{\bullet+1}(X_{n+1},A_{n+1}) \arrow[r]\arrow[d,"\delta^{\bullet+1}"] & \cdots \\
&  \vdots & \vdots&
\end{tikzcd}\]
where $\delta$ indicates the cohomology differentials and the maps $\tau$ are defined above.
Since $\psi^*_n$ and $\phi_n\circ -$ commute with differentials, the diagram is commutative, so that each $\tau_n$ induces a map on cohomology (which we will still denote by $\tau_n$).
Thus we obtain the inductive systems of cohomology groups:
\[H^\bullet_T(X_0,A_0) \xrightarrow{\tau_0}  H^\bullet_T(X_1,A_1) \xrightarrow{\tau_1} \cdots \xrightarrow{\tau_{n-1}}  H^\bullet_T(X_n,A_n) \xrightarrow{\tau_n} \cdots  \]
from which we derive their inductive limits: $\varinjlim H^\bullet_T(X_n,A_n)$.

\begin{sloppypar}
Our next goal is to relate $H^\bullet(X_n,A_n)$ to $H^\bullet_{TC}(\varprojlim X_k, \varinjlim A_l)$. More precisely we will show that $\varinjlim H^\bullet(X_n,A_n)$ is isomorphic to the continuous cohomology group $H^\bullet_{TC}(\varprojlim X_k, \varinjlim A_l)$ when each $A_l$ is a discrete group and the topologies on $\varprojlim X_k$ and $ \varinjlim A_l$ are the projective and inductive limit topologies.
\end{sloppypar}

Let us construct first a morphism from $\varinjlim H^\bullet_T(X_n,A_n)$ to $H^\bullet_{TC}(\varprojlim X_k, \varinjlim A_l)$ using the universal property of direct limits.
Consider the following diagram, corresponding to any representative $f$ of a class $[f]\in H^\bullet_T(X_n,A_n)$:
\[\begin{tikzcd}
(\varprojlim X_k)^\bullet \arrow[r,dashed]\arrow[d,"\pi^\bullet_n"] & \varinjlim A_l \\
X_n^\bullet \arrow[r,"f"] & A_n \arrow[u,"\iota_n"]
\end{tikzcd}\]
where $\pi^\bullet_n$ is the canonical projection $(\varprojlim X_k)^\bullet \rightarrow X^\bullet_n$ and $\iota_n : A_n \rightarrow \varinjlim A_l$ is the natural morphism of $A_n$ into the direct limit.
Observe that if $\iota_n f \pi_n^\bullet$ is %obtained
as above, then it factors through $X^\bullet_n$ %some $X^\bullet_k$
by definition and, in particular, any preimage of a subset in $\varinjlim A_l$ is a basis element of the topology of $\varprojlim X_k^\bullet$ since each $X_k$ is a discrete topological space and $\varprojlim X_k$ is endowed with the projective limit topology. Since $\varinjlim A_l$ is discrete, being a direct limit of discrete spaces; it follows then that $\iota_n f \pi_n^\bullet$ is continuous. Also, the correspondence $\{f:X_n^\bullet \rightarrow A_n\} \stackrel{\iota_n f \pi_n^\bullet}{\longrightarrow} \{ (\varprojlim X_k )^\bullet \rightarrow\varinjlim A_l \} $ respects equivalence classes, so it induces a well defined map $\sigma_n:H^\bullet_T(X_n,A_n)\rightarrow H^\bullet_{TC}(\varprojlim X_k, \varinjlim A_l)$.
We obtain
therefore the following diagram:
\[
\begin{tikzcd}
\cdots \arrow[r] & H^\bullet_T(X_{n-1},A_{n-1} ) \arrow[r] \arrow[dr,"\sigma_{n-1}"]& H^\bullet_T(X_n,A_n)\arrow[r]\arrow[d,"\sigma_n"] & H^\bullet_T(X_{n+1}, A_{n+1}) \arrow[r]\arrow[dl,"\sigma_{n+1}"] & \cdots\\
& & H^\bullet_{TC}(\varprojlim X_k, \varinjlim A_l) & &
\end{tikzcd}
\]
which can be seen to be commutative by a direct computation.
 By the universal property of colimits we obtain a unique morphism $\varinjlim H^\bullet_T(X_n,A_n)\xrightarrow{\Psi}H^\bullet_{TC}(\varprojlim X_k, \varinjlim A_l)$.
 In fact, $\Psi$ is an isomorphism.
 We include a proof to illustrate the computations that will be given in corollaries below for specific quandles.

\begin{pro}
     Suppose $(X_n,\psi_n)$ is a projective system of discrete quandles and $(A_m,\phi_m)$ is a direct system of discrete abelian groups. Then the map $\Psi$ defined above is an isomorphism.

\end{pro}
\begin{proof}
\begin{sloppypar}
Suppose $\alpha \in \varinjlim H^\bullet_T(X_n,A_n)$ is mapped to zero by $\Psi$. By construction of $\Psi$, it means that there exists some $i \in \mathbb{N}$ such that $[\iota_i f \pi_i^\bullet]$ is the zero class in $H_{TC}^\bullet(\varprojlim X_k, \varinjlim A_l)$, where $f$ is a representative of a class $[f] \in H^\bullet_T(X_i,A_i)$, $\pi_i^\bullet$ is the projection on the $i^{th}$ factor	and $\iota_i$ is the natural map $A_i \rightarrow \varinjlim A_l$. So $\iota_i f \pi_i^\bullet$ is a coboundary of some continuous $g : (\varprojlim X_k)^{\bullet-1} \rightarrow \varinjlim A_l$ which  factors through some $X^{\bullet-1}_j$, $j\in \mathbb{N}$. Choosing $t\in \mathbb{N}$ large enough it follows that $[\iota_t f \pi_t^\bullet] = 0$ in $H^\bullet_T(X_t,A_t)$.  Since $[\iota_i f \pi_i^\bullet]\sim[\iota_t f \pi_t^\bullet]$ in $\varinjlim H^\bullet_T(X_n,A_n)$ it follows that $\alpha$ is the zero class.
\end{sloppypar}

Suppose now we are given a class $[\beta] \in H_{TC}^\bullet(\varprojlim X_k, \varinjlim A_l)$. Since $\beta$ is continuous and $\varinjlim A_l$ is discrete, it factors through some $X_i^\bullet$. Therefore its image in $\varinjlim A_l$ is finite and it will be contained in some $A_j$. Choosing $t\in \mathbb{N}$ large enough, we get that $[\beta]$ is the image $T\overline{[h]}$, for some $[h] \in H^\bullet_T(X_t,A_t)$, where the bar symbol indicates that we are considering a representative class in $ \varinjlim H^\bullet_T(X_n,A_n)$.
	\end{proof}

    See \cite{Wilson} for an analogous statement for group cohomology.

Observe that the construction of the morphism $\Psi$ and the proof above are still valid if we consider topological compact Hausdorff quandles $X_n$ and replace each $H^\bullet_T(X_n,A_n)$ by their continuous counterparts, in virtue of Lemma~\ref{lem:cptH}.

	Fix an odd prime $p\in\mathbb{Z}$ and choose $u\in\mathbb{Z}$ such that $(u, p)=1$. Then multiplications by $u$ and by $1-u$ define automorphisms of
    $\mathbb{Z}/p^n\mathbb{Z}$, for any $n\in\mathbb{N}$.
    Thus we obtain an Alexander quandle structure on $\mathbb{Z}/p^n\mathbb{Z}$, which will be denoted by $X^u_{n}$.
	Also recall that there are natural maps $ \mathbb{Z}/p^n\mathbb{Z} \longrightarrow \mathbb{Z}/p^{n+1}\mathbb{Z}$ given by multiplication by $p$, which commute with the action by $\mathbb{Z}$ and consequently a direct system, whose direct limit is the Pr\"ufer group $\mathbb{Z}(p^\infty)$.
    Thus we have an  Alexander quandle
    $(\mathbb{Z}(p^\infty), u)$, denoted also by  $\mathbb{Z}(p^\infty)_u$.

	\begin{cor}
		$H^1_{TC}(\varprojlim X^u_n,\mathbb{Z}(p^\infty)_u) \cong  \mathbb{Z}(p^\infty)_u \times \mathbb{Z}(p^\infty)_u$
		for any  $u\in\mathbb{Z}$ such that $(u, p)=1$.
	\end{cor}
\begin{proof}
	As in Section 4.1 it is possible to show that the first twisted (discrete) cohomology group $H_T^1(X^u_n, \mathbb{Z}/p^n\mathbb{Z})$ is  the abelian group of affine maps $\{f_{\alpha,\beta}:\mathbb{Z}/p^n\mathbb{Z}\rightarrow\mathbb{Z}/p^n\mathbb{Z}\mid f(x) = \alpha x + \beta,\ \alpha,\beta\in \mathbb{Z}/p^n\mathbb{Z}\}$ which can be seen to be isomorphic to $\mathbb{Z}/p^n\mathbb{Z}\times \mathbb{Z}/p^n\mathbb{Z}$ for all $n\in\mathbb{N}$, via the isomorphism $f_{\alpha,\beta}\mapsto\alpha\times\beta$.
	Using the definition of $\sigma_n:H^1_T(X^u_n,\mathbb{Z}/p^n\mathbb{Z}) \rightarrow H^1_T(X^u_{n+1},\mathbb{Z}/p^{n+1}\mathbb{Z})$ we have that $(\sigma_nf) (x) = (pf\pi_n) [x] = p\alpha [x] + p\beta$, from which we obtain the direct limit
	\[
	\mathbb{Z}/p\mathbb{Z}\times \mathbb{Z}/p\mathbb{Z} \rightarrow \cdots \rightarrow \mathbb{Z}/p^n\mathbb{Z}\times\mathbb{Z}/p^n\mathbb{Z} \rightarrow \cdots
	\]
	where each map is just multiplication by $p$ on each coordinate. \\
	It follows that $\varinjlim H^1_T(X^u_n,\mathbb{Z}/p^n\mathbb{Z}) = \mathbb{Z}(p^\infty)_u\times\mathbb{Z}(p^\infty)_u$ and the result follows.
\end{proof}

\begin{cor} $H_C^3(\varprojlim R_{p^n}, \mathbb{Z}/p\mathbb{Z}) = 0$, where $R_{p^n}$ denotes the dihedral quandle on $p^n$ elements and the cohomology group is meant to be untwisted.
	\end{cor}
\begin{proof}
\begin{sloppypar}
	For a given odd prime, it has been computed by Mochizuki \cite{Mo}, that $H^3(R_{p^n}, \mathbb{Z}/p\mathbb{Z}) = \mathbb{Z}/p\mathbb{Z}$, where $R_{p^n}$. Directly from the proof in \cite{Mo}  it also follows that the map $H^3(R_{p^n}, \mathbb{Z}/p\mathbb{Z}) \rightarrow H^3(R_{p^{n+1}}, \mathbb{Z}/p\mathbb{Z})$ induced by the canonical projection $ R_{p^{n+1}} \rightarrow R_{p^n}$ is the trivial map. We obtain the inductive system:
    \end{sloppypar}
	\[
	\mathbb{Z}/p\mathbb{Z} \xrightarrow{0} \mathbb{Z}/p\mathbb{Z} \xrightarrow{0} \cdots \xrightarrow{0} \mathbb{Z}/p\mathbb{Z} \xrightarrow{0} \cdots
	\]
	whose direct limit is the trivial group.
\end{proof}

\appendix

\section{Further Non-trivial Cohomology}\label{sec:nontriv}

In this section we present a generalization of Proposition~\ref{prop:linear} that provides further examples of non-trivial cohomology groups.

\begin{pro}
	Let $X=({\mathbb R}^n, S)$ and $A=({\mathbb R}^m, T)$ be  Alexander quandles, where $S \in {\rm GL}_n(\mathbb{R})$ and
	$T \in {\rm GL}_m(\mathbb{R})$, respectively.
	Then $H^2_{\rm TC} ( X, A ) \neq 0$ if  the following conditions hold for $k>1$:
	$\sum_{i=0}^{k+1} (-S)^i = 0  = \sum_{i=0}^{k+1} (-T)^i$, and there exists an $n \times m$ matrix $C$
	such that $\sum_{\ell=0} ^{k} (-T)^\ell C (  \sum_{j=1}^{k-\ell +1} (-S)^j ) \neq 0$.
\end{pro}

\begin{proof}
	Let $w=\sum_{i=0}^{k} T^i (u_i, v_i) $.
	One computes
	\begin{eqnarray*}
		\partial w &=&
		\sum_{i=0}^{k} T^i [ T (u_i) + (1-T) ( v_i ) - ( S u_i + (1-S) v_i ) ] \\
		&=&
		[ ( v_0 ) - ( S u_0 + (1-S) v_0 ) ] \\
		& & +
		\sum_{i=1}^{k-1} T^i [ (u_{i-1}) - (v_{i-1}) + ( v_i )  - ( S u_i + (1-S) v_i ) ] \\
		& & +
		T^{k+1} [ ( u_k ) - (v_k) ] .
	\end{eqnarray*}
	By setting
	\begin{equation} \label{eqn:ui}
	v_1=S u_0 + (1-S) v_0 ,  \quad u_{i-1}= S u_{i} + (1-S) v_i , \quad  v_j=v_{j-2},   \quad  {\rm and}  \quad u_{k}=v_{k-1}
	\end{equation}
	for $i=1, \ldots, k$ and $j=2, \ldots, k$, we obtain
	$$ \partial w = ( \sum_{\ell=0}^{k+1} (-T)^\ell )[  (v_0) - (v_1) ]. $$
	Hence the condition (\ref{eqn:ui}) and the assumption $  \sum_{\ell=0}^{k+1} (-T)^\ell =0$ implies $\partial w=0$.

	For $k$ odd, set
	\begin{eqnarray*}
		u_{k-2i} &=&  ( \sum_{j=2}^{2i+1} (-S)^j ) u_0 + (1 -  \sum_{j=2}^{2i+1} (-S)^j ) v_0 \\
		u_{k-(2i+1)} &=&  ( - \sum_{j=1}^{2i+2} (-S)^j ) u_0 + ( \sum_{j=0}^{2i+2} (-S)^j ) v_0
	\end{eqnarray*}
	and for even $k$, set
	\begin{eqnarray*}
		u_{k-2i} &=&  ( - \sum_{j=1}^{2i+1} (-S)^j ) u_0 + ( \sum_{j=0}^{2i+1} (-S)^j ) v_0 \\
		u_{k-(2i+1)} &=&  ( \sum_{j=2}^{2i+2} (-S)^j ) u_0 + ( 1- \sum_{j=2}^{2i+2} (-S)^j ) v_0 .
	\end{eqnarray*}
	Then it is checked by induction that these satisfy Equations (\ref{eqn:ui}).

	For $\phi( x_1, x_2 ) = C(x_1 -x_2 )$ as in Proposition~\ref{prop:linear}, one computes
	$$\phi (w) = \sum _{\ell =0} ^{k}
    \phi(T^\ell (u_\ell, v_\ell) )
    = \sum _{\ell =0} ^{k} T^\ell C (u_\ell - v_\ell)
    =
    \sum_{\ell =0} ^{k} (-T)^\ell C (  \sum_{j=1}^{k-\ell +1} (-S)^j ) (u_0 - v_0)  $$
	as desired. The last equality is obtained by substituting the formulas for $u_{k - 2i}$
    and $u_{k-(2i+1)}$ for each case of $k$ odd and even.
\end{proof}

\section{Continuous Isomorphisms of Topological Quandles}\label{ContinIsom}

In this section we point out that continuous isomorphism classes differ from algebraic isomorphism classes for
topological quandles.
The results of this section were obtained by W.~Edwin Clark.

\begin{lem}\cite{Reem} \label{lem:reem}
If $T: {\mathbb R}^n \rightarrow {\mathbb R}^m$ is additive and continuous, then $T$ is linear.
\end{lem}

We recall \cite{Sam} that a generalized Alexander quandle $({\mathbb R}^n, T)$
is indecomposable
if and only if $(I-T)$ is invertible.

\begin{lem}%[W.~Edwin Clark]
\label{lem:similar}
Let $({\mathbb R}^n, S)$ and $({\mathbb R}^m, T)$
be  topological Alexander quandles,
such that  $I -S$ and $I-T$ are invertible.
Let $F: {\mathbb R}^n  \rightarrow {\mathbb R}^m$ be a continuous quandle homomorphism such that $F(0)=0$.
Then $S, T, F$ are linear and  the condition $FS=TF$ holds.
\end{lem}

\begin{proof}
First from Lemma~\ref{lem:reem}, $S$ and $T$ are linear.
Since $F$ is a quandle homomorphism,
\begin{equation}\label{eqn:qhom}
F(S x + (I-S) y ) = T F(x) + (I-T) F(y)
\end{equation}
holds for all $x, y \in {\mathbb R}^n$.
By setting $x=0$ and $y=0$ respectively in Equation(\ref{eqn:qhom}) and using the assumption $F(0)=0$, we obtain
\begin{equation}\label{eqn:multi1}
F((I-S) y ) =  (I-T) F(y),
\end{equation}
and
\begin{equation} \label{eqn:multi2}
F(Sx) = TF(x),
\end{equation}
which is the condition $FS=TF$.
These Equations (\ref{eqn:multi1}) and (\ref{eqn:multi2}) also imply
\begin{equation}\label{eqn:add}
F(S x + (I-TS) y ) =  F( S x) +  F((I-S) y) .
\end{equation}
By the invertibility assumptions, we have
$\{ Sx \ | \ x \in {\mathbb R}^n \}= {\mathbb R}^n$ and
$\{ (I-S) y \ | \ y \in {\mathbb R}^n \}= {\mathbb R}^n$.
Hence Equation (\ref{eqn:add}) implies that
$F(a+b)=F(a)+F(b)$ for all $a, b \in {\mathbb R}^n$.
Since $F$ is additive and continuous, Lemma~\ref{lem:reem} implies that $F$ is linear.
\end{proof}

Solving the matrix equation $FS=TF$ can be found, for example,
in \cite{Barnnett}.
A direct calculation gives the following lemma.

\begin{lem} \label{lem:shift}
Let $({\mathbb R}^n, S)$ and $({\mathbb R}^m, T)$
be Alexander quandles.
Let $F: {\mathbb R}^n  \rightarrow {\mathbb R}^m$ be a  quandle homomorphism.
Let $a \in {\mathbb R}^m$. Then $F+a: {\mathbb R}^n  \rightarrow {\mathbb R}^m$ defined by
$(F+a)(x) = F(x) + a$ for $x \in {\mathbb R}^n$ is a quandle homomorphism.
\end{lem}

\begin{pro}%[W.~Edwin Clark]
\label{prop:similar}
Let $({\mathbb R}^n, S)$ and $({\mathbb R}^n, T)$
be  indecomposable topological Alexander quandles.
If $F: {\mathbb R}^n \rightarrow {\mathbb R}^n$ is a continuous quandle isomorphism such that $F(0)=0$, then $S, T, F$  are linear and $S$ and $T$ are similar:
$T=FSF^{-1}$.
\end{pro}

\begin{proof}
By Lemma~\ref{lem:reem}, $S, T, F$ are linear.
By Lemma~\ref{lem:similar}, $S$ and $T$ are similar via $F$.
\end{proof}

\begin{pro}%[W.~Edwin Clark]
\label{prop:noniso}
There is a family with continuum cardinality of topological quandle structures on ${\mathbb R}^n$ for all $n>0$, such that its elements
 are pairwise non-isomorphic as topological quandles but are isomorphic as algebraic quandles.

\end{pro}

\begin{proof}
Let $Q(u)$ be the field of rational functions over ${\mathbb R}$ with variable $u$.
Let $s \in {\mathbb R}$ be a transcendental.
Let $Q(u)$ act on ${\mathbb R}^n$ by the scalar multiple $f(u) \cdot x = f(s) x$.
Let $s, t$ be distinct transcendentals.
Then there are two vector space structures on ${\mathbb R}^n$ over $Q(u)$ by multiples by $s$ and $t$.
They have the same dimension as vector spaces, and therefore, there is a vector space isomorphism
$F: {\mathbb R}^n \rightarrow {\mathbb R}^n$ over $Q(u)$, and it satisfies $F(sx)=t F(x)$.
Hence $F$ is a quandle isomorphism.
If $F$ is continuous, then $F$ is linear over ${\mathbb R}$ by Lemma~\ref{lem:similar},
and $Fs=sF=tF$ and $s=t$, a contradiction.
Hence $F$ is a quandle isomorphism that is not continuous.
\end{proof}

\noindent
{\bf Acknowledgement.}
We thank W.~Edwin Clark, Xiang-dong Hou and Takefumi Nosaka
for valuable comments.
MS was partially supported by
NIH R01GM109459.

\end{document}